\documentclass[letterpaper, 10 pt, conference]{ieeeconf}  % Comment this line out
\IEEEoverridecommandlockouts                              % This command is only
\usepackage{amsmath,amssymb}
\usepackage[hidelinks]{hyperref}
\usepackage{xcolor}
\usepackage{graphicx}
\usepackage{booktabs}
\usepackage{dsfont}
\usepackage{algorithm}
\usepackage{algpseudocode}% http://ctan.org/pkg/algorithmicx
\usepackage[style=ieee,backend=bibtex,url=false,doi=false,isbn=false,date=year,citestyle=numeric-comp,maxbibnames=10,minbibnames=10,maxcitenames=10,mincitenames=10]{biblatex}

\usepackage{xcolor} 
\bibliography{paper}

\newtheorem{ass}{Assumption}
\newtheorem{thm}{Theorem}

\newtheorem{rem}{Remark}
\newtheorem{lem}{Lemma}

\newcommand{\Id}[1]{\color{red}  \color{black}}

\title{\LARGE \bf
	An essentially decentralized interior point method for control
}

\author{Alexander Engelmann, Gösta Stomberg and Timm Faulwasser% <-this % stops a space
	\thanks{The authors are  with the Institute for Energy Systems, Energy Efficiency and Energy Economics, TU Dortmund University, Dortmund, Germany.
		{\tt\small goesta.stomberg@tu-dortmund.de, \{alexander.engelmann, timm.faulwasser\}@ieee.org}}%
}

\begin{document}

	\maketitle
	\thispagestyle{empty}
	\pagestyle{empty}

	%%%%%%%%%%%%%%%%%%%%%%%%%%%%%%%%%%%%%%%%%%%%%%%%%%%%%%%%%%%%%%%%%%%%%%%%%%%%%%%%
	\begin{abstract}
		Distributed and decentralized optimization are key  for the control of networked systems.
		Application examples include  distributed model predictive control and  distributed sensing or estimation. 
		Non-linear systems, however, lead to  problems with non-convex constraints for which classical decentralized optimization algorithms  lack convergence guarantees. 
		Moreover, classical decentralized algorithms usually exhibit only linear convergence. 
		This paper presents an essentially decentralized  primal-dual interior point method with convergence guarantees for non-convex problems at a {super}linear rate. 
		We  show that the proposed method works reliably on a numerical example from power systems.  
		Our results indicate that the proposed method outperforms ADMM in terms of computation time and  computational complexity of the subproblems.       
		
	\end{abstract}
	
	\section{Introduction}
	
	Distributed and decentralized optimization algorithms are of great interest for the control of  multi-agent systems.\footnote{\label{fn:dec} Throughout the paper, we refer to an optimization algorithm as \textit{distributed} if it contains a (preferably cheap) coordination problem solved by a central entity/coordinator. We denote an optimization algorithm as \textit{decentralized} in absence of such a coordinator and the agents rely purely on neighbor-to-neighbor communication \cite{Bertsekas1989,Nedic2018}. 
		We call an algorithm essentially decentralized if it has no central coordination but requires a small amount of central communication.
		Note that the definition of distributed and decentralized control differs~\cite{Scattolini2009}. }
		Many of these systems are governed by non-linear dynamics, which lead to optimization problems with non-convex  constraints.
		Classical decentralized optimization algorithms such as ADMM or Dual Decomposition, however, lack convergence guarantees for problems with non-convex constraints. 
		One of the few exceptions is the recently proposed Augmented Lagrangian Alternating Direction Inexact Newton (ALADIN) algorithm~\cite{Houska2016}. 
		ALADIN is a distributed algorithm that solves a centralized Quadratic Program~(QP) in its coordination step.
		Bi-level ALADIN~\cite{Engelmann2020c} decentralizes the solution of this QP and is therefore one of the few decentralized algorithms, which can provably solve non-convex problems with local convergence guarantees.
		However,  advanced decentralized globalization and numerical stabilization are not yet fully explored for~ALADIN.
	
	Interior point methods are successfully applied in centralized non-convex optimization and they have shown to solve large-scale problems reliably.
	Hence, attempts have been made to distribute interior point methods to transfer some of their properties to the distributed and decentralized setting. 
	These approaches typically employ Schur-complement techniques,  which shift main computational steps of the linear algebra to distributed computing architectures, see~\cite{Zavala2008,Chiang2014,Word2014}.
	Moreover, a primal barrier method has been proposed in~\cite{Bitlislioglu2017a}. 
	Therein, the authors consider non-convex \emph{in}equality constraints.
	The method solves a linear system of equations for coordination in a central coordinator.
	
	In the references above \emph{the} obstacle for decentralization is a system of linear equations, which couples the subproblems. 
	Solving this system of linear equations  via a coordinator renders these algorithms distributed but {not} decentralized. 
	The authors of \cite{Kang2014} solve the linear system iteratively by means of a conjugate gradient~(CG) method.
	However, \cite{Kang2014} does {not}  exploit sparsity in the Schur complement, i.e., the CG method is executed in a centralized fashion.
	Hence, there seems to be a gap in terms of decentralized interior point methods.

	This paper addresses this gap. 
	Similar to the references mentioned, we apply the Schur complement to reduce the dimension of the system of linear equations.
	We solve this system by means of essentially decentralized inner algorithms.
	Finite termination of inner  algorithms leads to  inexact solutions.
	Thus, we use ideas from inexact Newton methods \cite{Byrd1998,Dembo1982} to  establish superlinear local convergence under mild assumptions. 
	{As a particular inner algorithm, we employ the essentially decentralized conjugate gradient method (d-CG) from \cite{Engelmann2021}. 
		However,  the proposed framework is not limited to d-CG---other algorithms such as decentralized ADMM may also be used. 
	{Since we require scalar communication for stepsize selection, we obtain an essentially decentralized  algorithm.}
	The proposed scheme requires one matrix inversion per subproblem and outer iteration only. 
Thus, the complexity in the subproblems is   lower in comparison to ALADIN and ADMM, which both require to solve (convex/non-convex) NLPs locally. 
	We draw upon an example from power systems illustrating  the numerical performance for medium-sized problems. }
	
		The paper is organized as follows: 
		\autoref{sec:PrbStat} introduces the problem formulation and briefly recalls basics of centralized interior point methods. 
In 	\autoref{sec:decIP} we show how to exploit partially separable structure in the problem formulation to reduce the dimension in the Newton step.
		Moreover, we recall how to solve this reduced system in a decentralized
		fashion by tailored algorithms. 
		Using these (inner) algorithms leads to inexact Newton steps---hence we derive upper bounds on this inexactness such that superlinear local convergence is still guaranteed in \autoref{sec:locConv}.
		In \autoref{sec:NumRes} we apply our method to a numerical case study from power systems.

	\section{Preliminaries and problem statement} \label{sec:PrbStat}
	
	We consider partially separable problems
	\begin{subequations} \label{eq:sepForm}
		\begin{align} 
		\min_{x_i,\dots,x_{|\mathcal{S}|}} \; \sum_{i \in \mathcal{S}} &f_i(x_i) \\
		\text{subject to }\quad  g_i(x_i)&=0, & \forall i \in \mathcal{S}, \label{eq:sepProbGi} \\
		h_i(x_i) &\leq 0, & \forall i \in \mathcal{S},\label{eq:sepProbHi} \\
		\sum_{i \in \mathcal{S}} A_ix_i &= b,\label{eq:consConstr}
		\end{align}
	\end{subequations}
	where $\mathcal{S}=\{1,\dots, |\mathcal{S}|\}$ denotes a set of subsystems/agents each of which is equipped with an objective function $f_i:\mathbb{R}^{n_i} \rightarrow \mathbb{R}$ and local equality and inequality constraints $g_i, h_i:\mathbb{R}^{n_i} \rightarrow \mathbb{R}^{n_{gi}},\mathbb{R}^{n_{hi}}$. The matrices $A_i \in \mathbb{R}^{n_c \times n_i}$ and the vector $b\in \mathbb{R}^{n_c}$  encode coupling between subsystems.

	\subsection{Primal-dual interior point methods}

	Next, we recall the basic idea of centralized primal-dual interior point methods.
	These methods reformulate \eqref{eq:sepForm} as a barrier problem, where the  constraints \eqref{eq:sepProbHi} are replaced by logarithmic barrier terms in the objective function and the ``tightness'' of the reformulation is adjusted by a barrier parameter~$\delta >0$.
	The barrier problem reads
	\begin{subequations} \label{eq:slackReform}
		\begin{align} 
		\min_{x_1,\dots,x_{|\mathcal{S}|},v_1,\dots,v_{|\mathcal{S}|}} \; \sum_{i \in \mathcal{S}} &f_i(x_i) - \mathds 1^\top \delta  \ln (v_i) \hspace{-1.5cm} \\
		\text{subject to }\quad  g_i(x_i)&=0, &  \forall i \in \mathcal{S}, \label{eq:eqCnstr}\\
		h_i(x_i) +v_i&= 0, \;\; v_i \geq 0,& \forall i \in \mathcal{S}, \label{eq:ineqCnstr}\\
		\sum_{i \in \mathcal{S}} A_ix_i &= b \label{eq:conCnstr},
		\end{align}
	\end{subequations}
	where  $\mathds{1}= (1,\dots,1)^\top \in \mathbb{R}^{n_{hi}}$, the $\ln(\cdot)$ is evaluated component-wise, and where $v_i\in \mathbb{R}^{n_{hi}}$ are slack variables.
	Problem \eqref{eq:slackReform} is typically solved for a decreasing sequence of the barrier parameter $\delta>0$.
	Note that for $\delta \rightarrow 0$, the original problem  \eqref{eq:sepForm} and \eqref{eq:slackReform} are equivalent.
	
	Interior point methods differ in the way they update the barrier parameter, whether or not they use slack variables and how they solve the barrier problem \eqref{eq:slackReform}, cf. \cite{Nocedal2006}. Subsequently, we use a variant that applies only \emph{one} Newton step to the optimality conditions of \eqref{eq:slackReform} and then updates~$\delta$.
	This way, solving the barrier problem becomes cheap, but the update rule  for $\delta$ has to be chosen carefully~\cite[Sec. 4]{Byrd1998}.

	\section{Decomposition of interior point methods} \label{sec:decIP}
	Next, we show how to decentralize the computation of Newton steps.
	
	\subsection{A Newton step on the barrier problem}
	Consider the Lagrangian to \eqref{eq:slackReform},
	\begin{align*}
	L(p)= \sum_{i \in \mathcal{S}} f_i(&x_i) -  \mathds{1}^\top \delta  \ln (v_i) +\gamma_i^\top  g_i(x_i)    \\
	&+ \mu_i^\top (h_i(x_i) + v_i) + \lambda^\top \left  (\sum_{i \in \mathcal{S}} A_ix_i-b \right ),
	\end{align*}
	with $p=(p_1,\dots,p_{|\mathcal S|},\lambda)$ and $p_i  = ( x_i, v_i, \gamma_i, \mu_i )$, where $\gamma_i$, $\mu_i$ and $\lambda$ are Lagrange multipliers assigned to \eqref{eq:eqCnstr}, \eqref{eq:ineqCnstr} and \eqref{eq:conCnstr} respectively.
	The first-order optimality conditions read 
	\begin{align} \label{eq:KKTcond}
	F^\delta(p) = 
	\begin{pmatrix}
	F_1^\delta (p_1,\lambda) \\
	\vdots \\
	F_{|\mathcal S|}^\delta (p_{|\mathcal S|},\lambda) \\
	\sum_{i \in \mathcal{S}} A_i x_i -b
	\end{pmatrix}
	=0, 
	\end{align}
	with
	\begin{align*}
	F_i^\delta &(p_i,\lambda)= \\
	&
	\begin{pmatrix}
	\nabla_{x_i} f_i(x_i) + \nabla_{x_i} g_i(x_i) ^\top \gamma_i + \nabla_{x_i} h_i(x_i) ^\top \mu_i + A_i^\top \lambda \\
	-\delta V_i^{-1} \mathds{1}  + \mu_i \\ 
	g_i(x_i) \\
	h_i(x_i) + v_i 
	\end{pmatrix},
	\end{align*}
	and $V_i=\operatorname{diag}(v_i)$.
	%Here, we multiplied the second row in each $F_i^\delta(p_i,\lambda)=0$ with $V_i$.
	An exact Newton step  $\nabla F^\delta(p)\Delta p = - F^\delta (p) $ applied to \eqref{eq:KKTcond} yields 
	\begin{align} \label{eq:structKKT}
	\begin{pmatrix}
	\hspace{-.3mm}\nabla F_1^\delta  &0& \dots& \tilde A_1^\top\hspace{-.3mm} \\
	0& \nabla F_2^\delta& \dots& \tilde A_2^\top  \hspace{-.3mm}\\
	\vdots & \vdots & \ddots& \vdots \\
	\tilde A_1 & \tilde A_2 & \dots&  0
	\end{pmatrix}
	\hspace{-1.9mm}
	\begin{pmatrix}
	\hspace{-.3mm}\Delta p_1 \hspace{-.3mm} \\
	\hspace{-.3mm}\Delta p_2\hspace{-.3mm} \\
	\vdots \\
	\Delta \lambda
	\end{pmatrix}
	\hspace{-1.4mm}
	=
	\hspace{-1.2mm} 
	\begin{pmatrix}
	-F_1^\delta \\
	-F_2^\delta   \\
	\vdots\\
	\hspace{-.6mm} b \hspace{-.3mm} -\hspace{-.5mm}  \sum_{i \in \mathcal{S}} \hspace{-.7mm} A_i x_i \hspace{-.2mm}
	\end{pmatrix}
	\end{align}
	with 
	\begin{align*}
	\nabla F_i^\delta =
	\begin{pmatrix}
	\nabla_{xx} L_i  & 0 & \nabla g_i(x_i)^\top  & \nabla h_i(x_i)^\top   \\
	0 &-V_i^{-1}  M_i & 0 & I  \\
	\nabla g_i(x_i) & 0 &0 & 0  \\
	\nabla h_i(x_i) & I & 0 & 0  \\
	\end{pmatrix} ,
	\end{align*}
	$M_i = \operatorname{diag}(\mu_i)$,
	and
	$%\begin{align*}
	\tilde A_i = 
	\begin{pmatrix}
	A_i & 0 & 0 & 0
	\end{pmatrix}$.
	Here we used that $\delta V_i^{-2} = V_i^{-1}(\delta V_i^{-1})$ and that by \eqref{eq:KKTcond} and the definition of $F_i^\delta$, $\delta V^{-1}_i = M_i$.
	Assume temporarily that  $\nabla F_i^\delta$ is invertible.
	Then, one can reduce the dimensionality of \eqref{eq:structKKT} by   solving the first $S$ block-rows in \eqref{eq:structKKT} for $\Delta p_i$. This way we obtain 
	\begin{align} \label{eq:delP}
	\Delta p_i = -  \left(\nabla F_i^\delta \right )^{-1}\left (F_i^\delta + \tilde A_i^\top \Delta \lambda  \right ) \text{ for all } i\in \mathcal{S}.
	\end{align}
	Inserting \eqref{eq:delP} into the last row of \eqref{eq:structKKT} yields the Schur-complement
	\begin{equation} \label{eq:SchurComp}
	\begin{aligned}
	\Bigg (\sum_{i \in \mathcal{S}}\tilde A_i& \left(\nabla F_i^\delta \right )^{-1}\tilde A_i^\top \Bigg ) \Delta \lambda  \\
	&= \left(\sum_{i \in \mathcal{S}} A_i x_i -  \tilde A_i \left(\nabla F_i^\delta \right )^{-1}F_i^\delta  \right) - b. 
	\end{aligned}
	\end{equation}
	Observe that once we have solved \eqref{eq:SchurComp} we can compute $\Delta p_1,\dots,\Delta p_{|\mathcal S|}$ locally in each subsystem based on $\Delta \lambda$ via backsubstitution into \eqref{eq:delP}.
	Consider 
	\begin{subequations} \label{eq:Schur}
		\begin{align} 
		S_i \doteq& \tilde A_i \left(\nabla F_i^\delta \right )^{-1}\tilde A_i^\top, \quad \text{and} \\
		s_i \doteq&  A_i x_i - \tilde A_i \left(\nabla F_i^\delta \right )^{-1}F_i^\delta -\dfrac{1}{|\mathcal{S}|} b.
		\end{align}
	\end{subequations}
	Then, equation \eqref{eq:SchurComp} is equivalent to
	\begin{align} \label{eq:schurComp}
\left (\sum_{i \in \mathcal{S}} S_i \right )\,\Delta \lambda - \sum_{i \in \mathcal{S}} s_i = S \Delta \lambda -s =0.
	\end{align}

	\subsection{Solving \eqref{eq:schurComp} via  decentralized inner algorithms}
	Next, we briefly recall how to solve \eqref{eq:schurComp} via  decentralized  algorithms---for details we refer to \cite{Engelmann2021,Engelmann2020c}.
	Observe that $S_i$ and $s_i$ are constructed by multiplying $\nabla F_i^\delta$ with $A_i$ from the left/right. 
	Hence, zero rows in $A_i$ directly carry over to zero rows and columns in $S_i$ and $s_i$.
	As each row of the matrices $A_i$ typically describes coupling between only \emph{two} subsystems, these rows are zero in most $A_i$ except for those two subsystems that are coupled to each other---i.e., except for the two neighboring subsystems.
	As a result, quantities related to one specific row of $A$ have to be communicated \emph{only} between the two subsystems assigned to this row in many iterative inner algorithms.
	Hence we introduce the set of consensus constraints assigned to subsystem $i\in \mathcal{S}$,
	\begin{align} \label{eq:Cset}
	\mathcal{C}_i = \{j \in \{1,\dots n_c\} \; | \; [A_i]_j \neq 0\},
	\end{align}
	where $[\,A\,]_j$ denotes the $j$th row of $A$.
	Thus, we can define the neighbors of subsystem $i\in \mathcal{S}$ by $\mathcal{N}_i \doteq \{ j\in \mathcal{S} \; |\; \mathcal{C}_i \, \cap \, \mathcal{C}_j \neq \emptyset \}$.
	For the sake of completeness, the appendix recalls an essentially decentralized conjugate gradient method from  \cite{Engelmann2021}, which exchanges information mainly between neighbors, and allows solving \eqref{eq:schurComp}.

	\subsection{The essentially decentralized interior point method}
	
	\begin{algorithm}[t]
		\caption{d-IP for  solving  \eqref{eq:sepForm}}
		\begin{algorithmic}[1]
			\State Initialization: $p_i^0$ for all $i \in \mathcal{S}$, $\delta^0,\lambda^0$, $\epsilon$ \label{stp:1}
			\While{$\|F^\delta(p^k)\|_\infty > \epsilon$} 
			\State compute $(S_i^k,s_i^k)$ locally via \eqref{eq:Schur}  \label{stp:3}
			\While{$\|r_i^\lambda\|_\infty > c_1 (\delta^k)^{\eta}$ for all $i \in \mathcal{S}$} \label{stp:4}
			\State iterate \eqref{eq:schurComp} via an (ess.)  decentralized  algorithm \label{stp:5}
			\EndWhile\label{euclidendwhile}
			\State compute $(\delta_i^k,\alpha_i^{p,k},\alpha_i^{d,k})$ locally via \eqref{eq:barrierUp} and \eqref{eq:stepUp}  \label{stp:7}
			\State compute  $(\delta^k,\alpha^{p,k},\alpha^{d,k})$ centrally\label{stp:8}
			\State compute $(\Delta p_i^k,p^{k+1}_i)$ locally via \eqref{eq:delP} and \eqref{eq:primDuUp}   \label{stp:9}
			\State $k \rightarrow k+1$
			\EndWhile\label{euclidendwhile}
			\State \textbf{return} $p^\star$%\Comment{The gcd is b}
		\end{algorithmic} \label{alg:d-IP}
	\end{algorithm} 
	
	The  essentially decentralized interior point (d-IP) algorithm is summarized in Algorithm~\ref{alg:d-IP}.
	Here, superscripts $(\cdot)^k$ denote the iteration index. 
	After initializing d-IP in line~\autoref{stp:1}, each subsystem $i\in \mathcal{S}$ computes its Schur-complements $(S_i^k,s_i^k)$ locally based on \eqref{eq:Schur} in line~\autoref{stp:3}.
	In line~\autoref{stp:4}, a decentralized inner algorithm solves \eqref{eq:schurComp}. Possible inner algorithms are d-CG or decentralized ADMM (d-ADMM) from \cite{Engelmann2021}.
	The inner algorithms typically achieve a finite precision  only.\footnote{Note that in case of d-CG, although being an iterative algorithm, an exact solution can be achieved in a finite number of iterations \cite{Engelmann2021}.} 
	To quantify this inexactness, we introduce the residual 
	\begin{align} \label{eq:redSystem}
	r^{\lambda,k} = 
	S^k \Delta \lambda^k-s^k .
	\end{align}
	As we will see in \autoref{sec:locConv}, solving \eqref{eq:schurComp} up to precision
	\begin{align}  \label{eq:termCondInner}
	\|r^{\lambda,k}\|_\infty \leq c_1 \left (\delta^{k} \right )^{\eta}, \quad \eta > 1
	\end{align}
	guarantees fast local convergence.
	Note that since \eqref{eq:termCondInner} uses the infinity norm, the inequality can be evaluated locally if each subsystem computes local components of $r^{\lambda,k}$, which we denote by $r^{\lambda,k}_i$. 
	
	Line~\autoref{stp:7} of Algorithm~\ref{alg:d-IP} updates the barrier parameter 
	\begin{align} \label{eq:barrierUp}
	\delta^{k+1} = \underset{i\in \mathcal{S}}{\max}\; \delta_i^{k+1} \;\; \text{with}\quad \delta_i^{k+1} =\theta \left  (\frac{v_i^{k\top} \mu_i^{k\phantom{\top}}}{n_{hi}} \right )^{1+\gamma }
	\end{align}
	for parameters $\gamma > 0$, and $\theta$ close to $1$.
	Moreover, line~\autoref{stp:9} computes step updates locally via
	\begin{subequations} \label{eq:primDuUp}
		\begin{align} 
		x^{k+1}_i &= x_i^k + \alpha^p \Delta x_i^k, \qquad & \mu^{k+1}_i &= \mu_i^k + \alpha^d \Delta \mu_i^k, \\
		s^{k+1}_i &= s_i^k + \alpha^p \Delta s_i^k, \qquad & \gamma^{k+1}_i &= \gamma_i^k + \alpha^d \Delta \gamma_i^k.
		\end{align} 
	\end{subequations}
The stepsizes are computed by $\alpha^{p,k} = \underset{i \in \mathcal{S} }{\min  }\, \alpha_i^{p,k}$, $\alpha^{d,k} = \underset{i \in \mathcal{S} }{\min  }\, \alpha_i^{d,k}$, where
	\begin{subequations}\label{eq:stepUp}
		\begin{align}
		\alpha^{p,k}_i &= \min \left ( {\tau^k} \min_{\Delta [s_i^k]_n < 0} \left ( -\frac{[v_i^k]_n}{\Delta [v_i^k]_n} \right) ,1\right ) , \\
		\alpha^{d,k}_i &=\min \left ( {\tau^k} \min_{\Delta [\mu_i^k]_n < 0} \left ( -\frac{[\mu_i^k]_n}{\Delta [\mu_i^k]_n} \right),1\right ),
		\end{align}
	\end{subequations}
and $\tau^k = 1-(\delta^k)^{\beta}$ with $ \beta > \gamma$.

	\begin{rem}[Global communication] \label{rem:com}
			Note that for computing $\min/\max$ in \eqref{eq:barrierUp} and \eqref{eq:stepUp}, each subsystem communicates three floats per  d-IP iteration. 
			This can be implemented via broadcasting protocols, where each subsystem receives all $(\alpha_i^{p,k}, \alpha_i^{d,k},\delta_i^k)$ and then calculates the $\min/\max$ locally. 
			Hence, we refer to the proposed algorithm as being \emph{essentially} decentralized instead of fully decentralized.    \hfill $\square$
	\end{rem}

	\section{Local convergence} \label{sec:locConv}
	Next, we analyze the local convergence properties of Algorithm~\ref{alg:d-IP}.
	To this end, we make the following assumption~\cite{Byrd1998}.
	\begin{ass}[Regularity of the local minimizer] \hfill \label{ass:basic}
		\begin{itemize} 
			\item The quadruple $p^\star=(x^\star,v^\star,\gamma^\star,\mu^\star)$ is a KKT point of~\eqref{eq:sepForm}, i.e. $F^0(z^\star)=0$, $\mu^\star \geq 0$, and $v^\star\geq 0$;
			\item the Hessian $\nabla_{xx}L(p)$ exists and is  Lipschitz at $p^\star$;
			\item LICQ\footnote{LICQ $=$ Linear Independence Constraint Qualification} holds  at $p^\star$;
			\item   $\nabla_{xx}L$ is positive definite (a strict version of SOSC\footnote{SOSC $=$ Second-Order Sufficient Condition});
			\item strict complementarity holds, i.e. $v^\star  + \mu^\star > 0 $.
		\end{itemize}
	\end{ass}
	The convergence analysis first considers the progress of the iterates towards a solution to the barrier problem~\eqref{eq:slackReform}, when performing an inexact Newton step via \eqref{eq:structKKT}. Second, we combine this estimate with a barrier parameter update. This way, we ensure that a solution to problem~\eqref{eq:slackReform} coincides with a solution to problem \eqref{eq:sepForm} in the limit.
	We start by recalling a result from \cite{Byrd1998}, which characterizes the progress of Algorithm~\ref{alg:d-IP} towards a local solution to  barrier problem~\eqref{eq:slackReform}.
	\begin{thm}[{Distance to barrier  solution, \cite[Thm. 2.3]{Byrd1998}}] \label{thm:distToBarr}
		Suppose  \autoref{ass:basic} holds, let $p^k$ be an iterate close  to $p^\star(\delta^k)$, where $p^\star (\delta^k)$ denotes a primal-dual solution to the barrier problem \eqref{eq:slackReform}.
		Then, there exists a $\bar \delta$ such that for all $\delta < \bar \delta$
		\begin{align*}
		\|p^{k+1}-p^\star(\delta)\| \leq KL\|p^k-p^\star(\delta)\|^2 + K\|r^{\lambda,k}\|
		\end{align*}
			with $K=\|\nabla F^{0}(p^k)^{-1}\|$ and where $\|\cdot \|$ refers to the Euclidean norm.
			Here, $L$ is a constant satisfying
			\begin{align*}
			\|\nabla F^0(p) (p-p') -F^0(p) +  F^0(p')\| \leq L \|p-p'\|^2
			\end{align*} 
			for all $p,p'$ close enough to $p^\star$.
			\hfill $\square$
	\end{thm}
	We remark that the constant $L$ is guaranteed to exist by \autoref{ass:basic}, cf. \cite[Lem. 2.1]{Byrd1998}.
	
	\begin{rem}[Residual to \eqref{eq:structKKT} vs. residual  \eqref{eq:redSystem}]
			Note that in~\cite{Byrd1998}, the authors use the full residual corresponding to~\eqref{eq:structKKT} instead of $r^{\lambda,k}$ corresponding to \eqref{eq:redSystem}. 
			However, both residuals have the same value since \eqref{eq:delP} and \eqref{eq:SchurComp} ensure that the first $|\mathcal{S}|$ block-rows in \eqref{eq:structKKT} have residual zero.  \hfill $\square$
		\end{rem}

	Next, we combine \autoref{thm:distToBarr} with an update rule decreasing the barrier parameter $\delta$ similar to \cite[Thm. 2.3]{Byrd1998}.
	%, where have to ensure that the combined iterates converge to a local minimizer of the original problem~\eqref{eq:sepForm}.
	%Due to space limitations, we use a simplified analysis using arguments from \cite[Thm 2.5]{Byrd1998} for showing R-linear convergence.
	\begin{thm}[Local superlinear convergence.] \label{thm:supConv}
		Suppose \autoref{ass:basic} holds and let $p^k$ be an iterate sufficiently close to $p^\star$.
		Moreover, suppose that the residual $r^{\lambda,k}$ satisfies \eqref{eq:termCondInner}	for some positive constant $c_1$, suppose that $\delta^{k+1} = o(\|F^0(p^k)\|)$, and suppose that  \eqref{eq:stepUp} is satisfied with $\alpha^{p,k}=\alpha^{d,k}=1$ for all subsequent $k$.
		Then, the iterates generated by Algorithm~\ref{alg:d-IP} converge  to $p^\star$ locally at a superlinear rate. \hfill $\square$
	\end{thm}
	\begin{proof}
		Let  $p^\star$ denote a primal-dual solution to~\eqref{eq:sepForm}.
		Define $\tilde F(p;\delta)\doteq F^\delta(p)$.
			Observe that \autoref{ass:basic} implies that the implicit function theorem can be applied to $\tilde F(p;\delta)=0$ at $(p^\star,0)$~\cite[Lem. 2.2]{Byrd1998}.
			Thus, for all $\delta < \bar \delta$ there exists a constant  $C=\underset{p \in \mathcal{N}(p^\star)}{\max} \|{\nabla_p \tilde F(p;\delta)}^{-1}  \nabla_\delta \tilde  F(p;\delta)\|$, for which
			$ 
			\|p^\star(\delta)-p^\star\| \leq C \delta.
			$
		Combining this with the triangular inequality, \autoref{thm:distToBarr},  and $v^{k+1},\mu^{k+1} > 0$ we obtain
		\begin{align*}
		\|p^{k+1} \hspace{-.9mm}-p^\star\| \hspace{-.9mm}&\leq \|p^{k+1}-p^\star(\delta^{k+1})\| + \|p^\star(\delta^{k+1})-p^\star\|\\
		&\leq  KL\|p^k-p^\star(\delta^{k+1})\|^2 + K\|r^{\lambda,k}\| + C\delta^{k+1} \\
		& \leq  KL(\|p^k-p^\star\|+\|p^\star\hspace{-.9mm}-p^\star(\delta^{k+1})\|)^2\\
		&\phantom{=} + (Kc_1(\delta^{k+1})^{\eta-1} \hspace{-.5mm} +  C)\delta^{k+1} \\
		& \leq  2KL(\|p^k-p^\star\|^2 \hspace{-.9mm}+ \hspace{-.9mm} \|p^\star \hspace{-.9mm}-p^\star(\delta^{k+1})\|^2) \\
		& \phantom{=}+ (Kc_1(\delta^{k+1})^{\eta-1} \hspace{-.5mm} + \hspace{-.5mm} C)\delta^{k+1} \\
		& \leq  2KL\|p^k-p^\star\|^2 \\
		& \phantom{=}+(2KLC^2\delta^{k+1}  + Kc_1 (\delta^{k+1})^{\eta-1} + C)\delta^{k+1}.
		\end{align*}
		Dividing by $\|p^k-p^\star\|$, using that $F^0$ is Lipschitz with constant $P$ by \autoref{ass:basic},  and leveraging that $F^0(p^\star)=0$, we obtain
		\begin{align*}
		\frac{\|p^{k+1}\hspace{-.9mm}-p^\star\|}{\|p^k-p^\star\|} & \leq  2KL\|p^k-p^\star\|\\
		& \hspace{-6mm}+(2KLC^2\delta^{k+1} + Kc_1(\delta^{k+1})^{\eta-1} + C)\frac{\delta^{k+1}}{\|p^k-p^\star\|} \\
		& \leq  2KL\|p^k-p^\star\| \\
		& \hspace{-6mm}+P(2KLC^2\delta^{k+1} \hspace{-1mm} + \hspace{-.5mm} Kc_1(\delta^{k+1})^{\eta-1} \hspace{-1mm} +\hspace{-.5mm} C)\frac{\delta^{k+1}}{\|F^0(p^k)\|}. 
		\end{align*}
		This shows superlinear convergence if $\delta^{k+1}=o(\|F^0(p^k)\|)$.
	\end{proof}
	
	\begin{rem}[Satisfying \eqref{eq:stepUp} automatically]
At first glance it is not clear how to ensure that $\alpha_i^{p,k}=\alpha_i^{d,k}=1$ in \autoref{thm:supConv}. 
However, with the given update rules for $\delta$ and $\tau$ it can be shown that  \eqref{eq:stepUp} is satisfied automatically in the area of local convergence and thus does not interfere with the convergence analysis given here \cite[Sec. 5]{Byrd1998}.
	\end{rem}

	\begin{rem}[Solving \eqref{eq:schurComp} via decentralized algorithms]
		Note that Algorithm~\ref{alg:d-IP} is guaranteed to converge with \emph{any} inner algorithm solving \eqref{eq:schurComp} as long as \eqref{eq:termCondInner} is satisfied.
		Hence, one may apply essentially decentralized conjugate gradients  from \cite{Engelmann2021} or any other decentralized algorithm. In essence, this shows that the proposed algorithms is closely related to the concept of bi-level distribution of optimization algorithms introduced in \cite{Engelmann2020c}. \hfill $\square$
	\end{rem}
	
	\subsection{Updating the barrier parameter}
	Next, we show that the barrier parameter \eqref{eq:barrierUp} satisfies $\delta^{k+1} = o(\|F^0(p^k)\|)$ such that we can apply \autoref{thm:supConv} to ensure local superlinear convergence.
		The line search \eqref{eq:stepUp} ensures $v^k,\mu^k>0$, and the definition of $F^0(p^k)$ implies $\|V^k\mu^k\| \leq \|F^0(p^k)\|$. 
		Combining this with \eqref{eq:barrierUp} yields 
		\begin{equation*}
		\frac{\delta^{k+1}}{\|F^0(p^k)\|} \leq \frac{ \max_{i \in \mathcal{S}}\theta (v_i^{k\top} \mu_i^k / n_{hi})^{1+\gamma}}{\|V^k\mu^k\|}.
		\end{equation*} 
	From the equivalence of norms it follows that
		\begin{equation*}
		\frac{\delta^{k+1}}{\|F^0(p^k)\|} \leq \frac{\theta \|V^k\mu^k\|_1^{1+\gamma}}{\min_{i\in \mathcal S}(n_{hi})^{1+\gamma} \|V^k\mu^k\|} \leq
		D\|V^k\mu^k\|^\gamma,
		\end{equation*} 
	for some  $D>0$. Since $\|V^k\mu^k\|^\gamma \rightarrow 0$ as $\|F^0(p^k)\| \rightarrow 0$, it follows that ${\delta^{k+1} = o(\|F^0(p^k)\|)}$ and we obtain local superlinear convergence.
	
	\section{A numerical case study} \label{sec:NumRes}

	\subsection{Distributed Optimal Power Flow}
	Optimal Power Flow (OPF) problems compute   generator set points in  electrical grids such that the total cost of power generation is minimized while all technical limits are met~\cite{Frank2016, Faulwasser2018}. 
	Distributed approaches are particularly important for OPF due to large grid dimensions and  due to the need for a limited information exchange between grid operators \cite{Molzahn2017}.
	
	The basic OPF problem reads
	\begin{subequations}\label{eq:OPF}
		\begin{align} 
		&\min_{s,v \in \mathbb{C}^{N}} \;f(s) \\
		\text{subject to} \quad s-&s^d  = \operatorname{diag} (v) Y v^*,  \label{eq:PFeq} \\
		\underline p \leq \operatorname{re}(s) &\leq \bar p, \quad
		\underline q \leq \operatorname{im}(s) \leq \bar q, \label{eq:pwBounds}\\
		\underline v \leq \operatorname{abs}(v) &\leq \bar v, \quad v^1 = v^s, \label{eq:vBounds}
		\end{align}
	\end{subequations}
	where the decision variables are complex voltages $v \in \mathbb{C}^N$ and complex power injections $s \in \mathbb{C}^N$  at all buses $N$.
	The objective function $f:\mathbb{C}^N\rightarrow \mathbb{R}$ describes the cost of power generation, which is typically quadratic in the active power injections $\operatorname{re}(s)$.
	Here, $\operatorname{re(\cdot)}$ and $\operatorname{im(\cdot)}$ denote the real part and imaginary part of a complex number; $(\cdot)^*$ denotes the complex conjugate.
	The grid physics are described via the power flow equations \eqref{eq:PFeq}, where $Y \in \mathbb{C}^{N\times N}$ is the complex bus-admittance matrix containing   topology and parameter information~\cite{Frank2016}. 
	Moreover, $s^d \in \mathbb{C}^N$ is the  given power demand at all buses. 
	The bounds \eqref{eq:pwBounds} describe technical limits on the power injection by generators and \eqref{eq:vBounds} models voltage limits.
	The second equation in \eqref{eq:vBounds} is a reference condition on the voltage at the first bus, $v^1$, where the complex voltage is constrained to a reference value~$v^s$. 
	
	There are various variants of the OPF problem differing in the choice of the coordinates (polar vs. rectangular), in the complexity of the equipment (with/without shunts/transformers), and in the cost functions.
	Here, we neglect line shunts and we use the polar form of the power flow equations.
	For details we refer to \cite{Frank2016,Molzahn2017}.
	Moreover, there are different ways of formulating \eqref{eq:OPF} to meet \eqref{eq:sepForm}---here we use the approach described in \cite{Engelmann2019}.
	As a numerical example, we use the IEEE 118-bus test system with data from  MATPOWER  \cite{Zimmerman2011}.
	The problem has a total number of $n_x=576$ decision variables, $n_g=470$ non-linear equality constraints and $n_h=792$ inequality constraints.

	\begin{figure}
		\includegraphics[trim=25 40 40 40,clip,width=\linewidth]{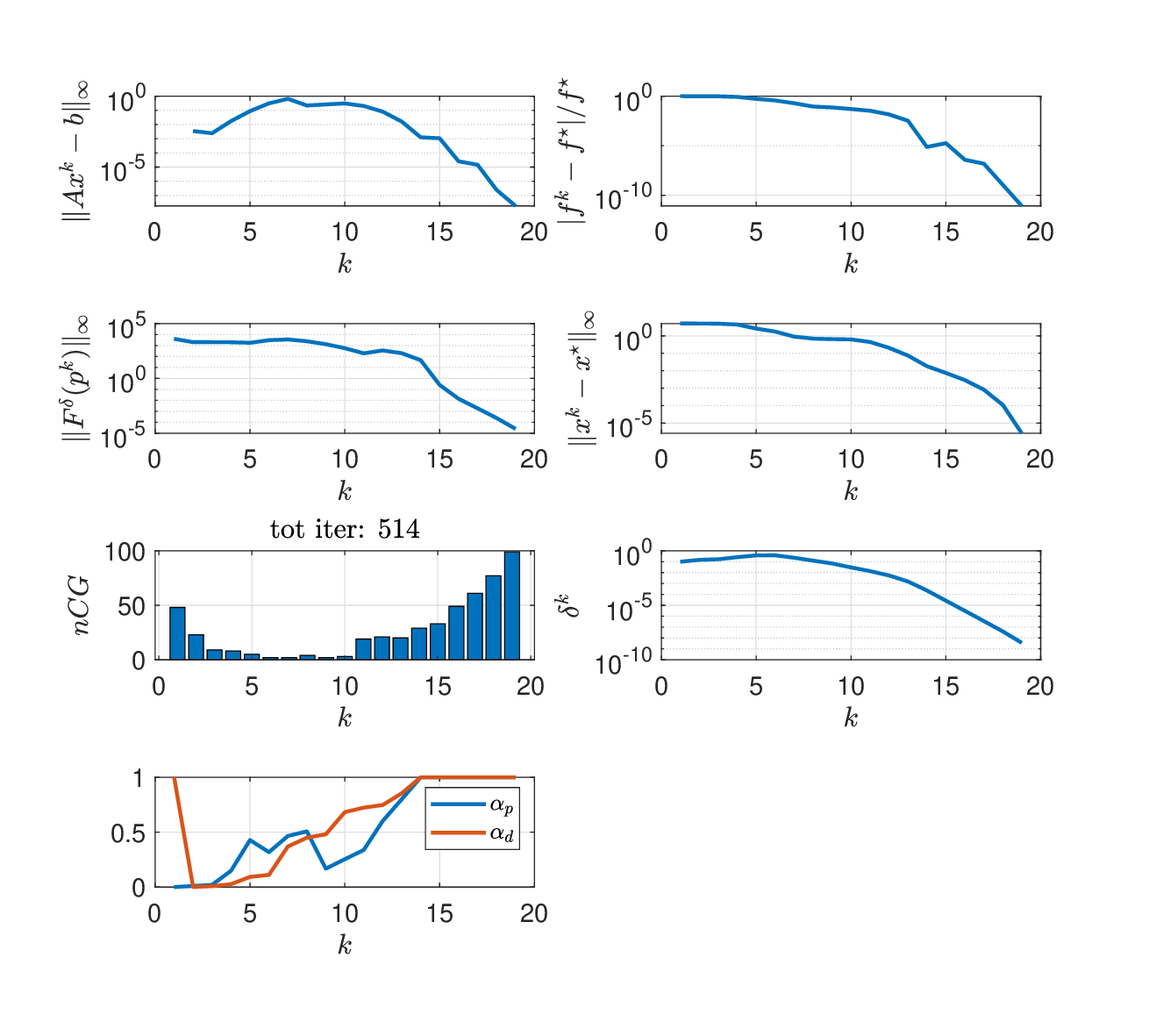}
		\caption{Convergence behavior of Algortihm~\ref{alg:d-IP} for  118-bus OPF.}
		\label{fig:OPF118bus}
	\end{figure}
	
	\subsection{Numerical results}
	We present numerical results for a prototypical implementation in MATLAB R2020b.
	We compare the numerical results of Algorithm~\ref{alg:d-IP} to ADMM, which is one of the most popular distributed optimization methods for OPF \cite{Erseghe2014,Guo2017}.
	We use the essentially decentralized conjugate gradient (d-CG) method from~\cite{Engelmann2021} as an inner algorithm.
	Since d-CG requires that $S^k$ is positive definite, we show that this is the case if \autoref{ass:basic} holds in Appendix~\ref{sec:posDefProof}. Moreover, we describe d-CG in Appendix~\ref{sec:DCG}.
	
	We distinguish between inner iterations, solving \eqref{eq:schurComp} via d-CG until \eqref{eq:termCondInner} is satisfied, and outer iterations, where the barrier parameter $\delta$, the stepsizes $(\alpha^p,\alpha^d)$, and the Schur complements $(S_i,s_i)$ are updated in all subsystems $i \in \mathcal{S}$. 
	
	\autoref{fig:OPF118bus} shows the convergence behavior of Algorithm~\ref{alg:d-IP} for the parameters $(c_1,\theta,\gamma,\beta,\eta)=(1,\,0.1,\,0.01,\,2,1.01)$.
	More specifically, it depicts the consensus violation $\|Ax^k-b\|_\infty$---i.e. the maximum mismatch of physical values at boundaries between subsystems---and the relative error in the objective $|f^k-f^\star|/f^\star$.
	Furthermore, \autoref{fig:OPF118bus} illustrates the error in the stationarity condition from \eqref{eq:KKTcond}, the distance to the minimizer $\|x^k-x^\star\|_\infty$, the number of inner iterations of d-CG, the barrier parameter sequence $\{\delta^k\}$, and the primal and dual step size $(\alpha^p,\alpha^d)$ from \eqref{eq:stepUp}.
	The bottom-left plot shows the stepsizes over the iteration index $k$.
	Here, one can see that the iterates enter the region of  local convergence after 14 outer iterations, after which only full steps are taken and superlinear convergence can be observed.
	The number of inner d-CG iterations required for satisfying \eqref{eq:termCondInner} vary greatly and more iterations are needed as d-IP approaches a local minimizer.
	This is due to the fact that $\delta^k$ approaches zero and thus high accuracies in \eqref{eq:schurComp} are required.

	\autoref{fig:OPF118ADMM} compares the results of Algorithm~\ref{alg:d-IP} to an ADMM implementation and to bi-level ALADIN from \cite{Engelmann2020c}.
	We use bi-level ALADIN with d-CG as an inner algorithm and a fixed amount of 50 inner iterations, which is the minimum amount of iterations required for convergence in this problem.
	One can observe that when counting outer iterations only (yellow line), d-IP is much faster compared to ADMM (blue lines).
	However, when counting inner d-CG iterations (orange line), d-IP requires about 400 and ADMM requires about 300  iterations for reaching an accuracy of $\|x^k-x^\star\|_\infty< 10^{-4}$.
	In our prototypical implementation, d-IP requires $7.6\,$s for solving \eqref{eq:OPF} to the above accuracy, whereas ADMM requires $39.9\,$s in the best-tuned case and bi-level ALADIN requires $16.1\,$s. 
	This is due to the fact that ADMM requires solving about 300 NLPs, whereas d-IP only requires to perform  15 matrix inversions for the outer iterations.
	Hence, the complexity per iteration of d-IP is much smaller compared to ADMM, which results in the reduced execution time.
	Bi-level ALADIN has the same complexity per iteration as ADMM. However, due to the smaller amount of outer iterations, the total computation time is much smaller than for ADMM but still larger than for d-IP.
	Moreover, d-IP and bi-level ALADIN are guaranteed to converge locally, whereas ADMM is not due to the non-convex constraints~\eqref{eq:PFeq}.
	An advantage of ADMM is its full decentralization.

	\begin{figure}
		\includegraphics[trim=30 20 50 15,clip,width=\linewidth]{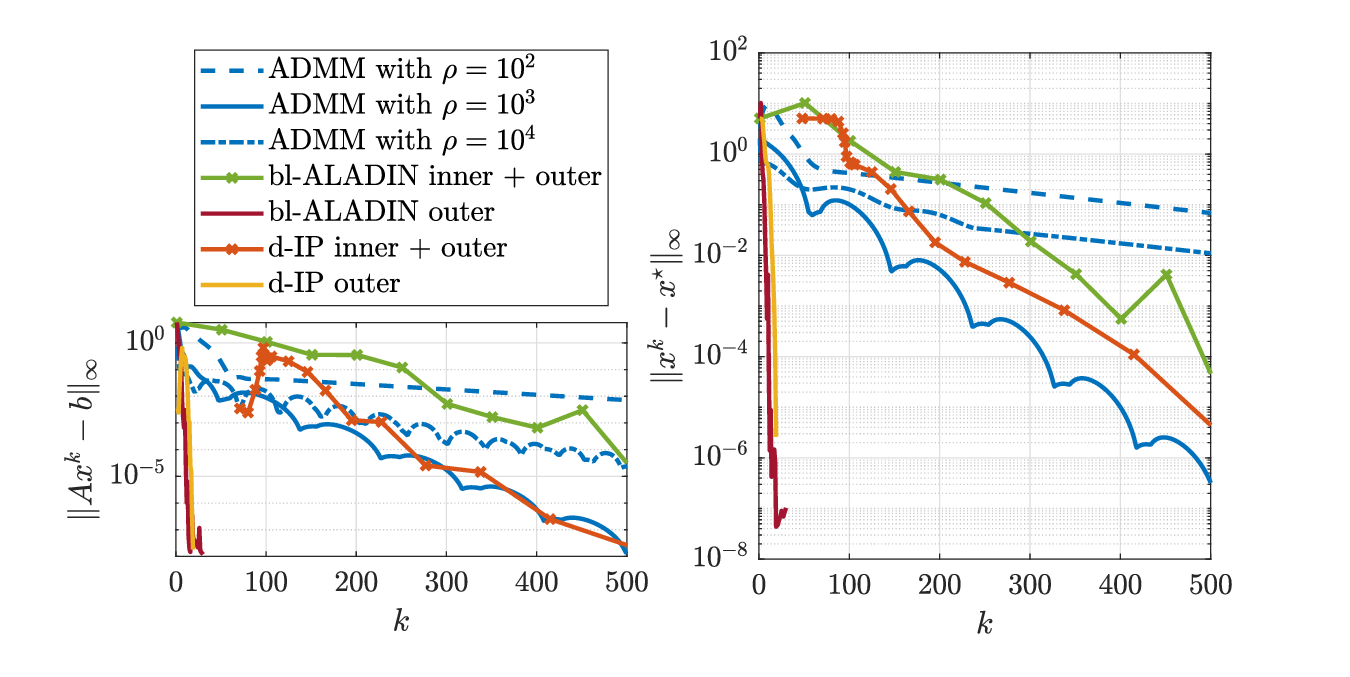}
		\caption{ADMM, d-IP, and bi-level ALADIN for 118-bus OPF.}
		\label{fig:OPF118ADMM}
	\end{figure}

	\subsection{Communication analysis}
		In contrast to ADMM, d-IP requires global scalar communication to sum $\sigma$ and $\eta$ in d-CG (Algorithm~\ref{alg:d-CG}), and the communication of three additional scalars every outer d-IP iteration for computing the barrier parameter update  \eqref{eq:barrierUp} and the step lengths \eqref{eq:stepUp}.
		However, note that these quantities can be communicated via broadcasting and thus this does not interfere much with decentralization, cf. \autoref{rem:com}.
				A quantitative comparison of the communication demands is omitted due to space limitations---for first results in this direction we refer to \cite{Stomberg2021a}.

	The communication demand and further important properties of d-IP and ADMM are summarized in~\autoref{tab:propIpADM}.
	
	\begin{table}
		\centering
		\footnotesize
		\caption{Properties of d-IP and ADMM for problem  \eqref{eq:sepForm}.}
		\begin{tabular}{rp{2.7cm}lc}
			\toprule
			&	d-IP & ADMM  \\
			\midrule
			conv. guarantees &  yes & no \\
			conv. rate  & superlinear & (sub)linear \\
			local complexity & matrix-vector product / (matrix inversion) & solving a NLP \\
			global comm. & scalar stepsize & none \\
			local comm. & vectors & vectors \\
			%		timing & ++ & -  \\
			\bottomrule
		\end{tabular}
		\label{tab:propIpADM}
	\end{table}
	
	\Id{Add that divergence of ADMM also occurs in practice?}
	
	\Id{Add something that if local complexity is of importance, d-IP is in any case better, if only communication matters, d-IP is comparable in some cases, but in some cases, the communication overhead can be much larger.}

	\Id{Remark that larger OPF also possible?}
	
	\section{Summary and Outlook}
	This paper has presented a novel essentially decentralized interior point method, which is based on neighbor-to-neighbor communication. The proposed method solves partially separable  NLPs with non-convex constraints to local optimality with fast local convergence guarantees.
	We illustrate the efficacy of the method drawing upon an example from power systems.
	Future work will consider constraint preconditioning and regularization techniques to improve the convergence rate of d-CG as an inner algorithm.
	Moreover, more elaborate decentralized globalization routines are of interest.
	Furthermore, weakening Assumption 1 to standard second-order conditions,  providing an open-source implementation and testing d-IP on large-scale problems seems promising.

	\appendix 
	
	\subsection{Essentially Decentralized Conjugate Gradients} \label{sec:DCG}
	For the sake of a self-contained presentation, we recall the essentially decentralized conjugate gradient (d-CG) algorithm  for solving~\eqref{eq:SchurComp} from  \cite{Engelmann2021}.
	Consider unit vectors
	\[
	e_i^\top\doteq (0,\;\dots,\underbrace{1}_{i\text{-th element}},\dots,0) \in \mathbb{R}^{n_c}
	\]
	and define the following projection matrices 
	\begin{align} \label{eq:Idef}
	I_{\mathcal{C}_i}\doteq\left (e_j^\top\right )_{l},\quad l \in \mathcal{C}_i\subseteq \{1,\dots,n_c\}.
	\end{align}
	Here $\mathcal{C}_i$ is from \eqref{eq:Cset} and $(v_i)_{i\in \mathcal{S}}$ denotes the vertical concatenation of vectors $v_i$ indexed by the set $\mathcal{S}$.
	Using these matrices we project the Schur complement matrices and vectors from~\eqref{eq:SchurComp} such that zero rows and columns are eliminated: 
	$
	\hat S_i = I_{\mathcal C_i}^{} S_i I_{\mathcal C_i}^{\top},  \hat {s}_i = I_{\mathcal{C}_i}s_i.
	$
	Moreover, we use $I_{\mathcal{C}_i}$ to construct matrices mapping between the subsystems, i.e. $I_{ij}\doteq I_{\mathcal{C}_i}^{\phantom \top} I_{\mathcal{C}_j}^{\top}$.
	Furthermore, we consider averaging matrices
	\begin{align} \label{eq:LamDef}
	\Lambda_i\doteq I_{\mathcal{C}_i}^{\phantom \top}  \left ( \sum_{j \in \mathcal{S}} I_{\mathcal{C}_j}^{ \top} I_{\mathcal{C}_j}^{\phantom \top} \right ) I_{\mathcal{C}_i}^{ \top}= 
	I_{\mathcal{C}_i}^{\phantom \top} \Lambda  I_{\mathcal{C}_i}^{ \top}.
	\end{align}
	This allows summarizing the essentially decentralized conjugate gradient method in Algorithm~\ref{alg:d-CG}. 
	The upper index $(\cdot)^n$ denotes d-CG iterates.

\begin{algorithm}[t]
	\small 
	\caption{Essentially decentralized CG  for problem \eqref{eq:SchurComp} }
	\textbf{Initialization for all $i \in \mathcal{S}$: $ \lambda^0_i$, $r_i^0=p_i^0= \sum_{j \in \mathcal{N}_i}I_{ij} \hat s_j - \sum_{j \in \mathcal{N}_i} I_{ij}^{\phantom \top} \hat S_j \lambda_j^0$,  $\eta^{0}_i= r_i^{0\top } \Lambda_i^{-1} r_i^0$ and $\eta^{0}= \sum_{i \in \mathcal{S}} \eta_i^{0}$.}\\
	\textbf{Repeat until $r_i^n= 0$ for all $i \in \mathcal{S}$:}
	\begin{subequations}
		\begin{align} 
			&\sigma^{n}_i = p^{n \top }_i \hat S_i p^{n}_i;\;\;  u_i^{n}= \hat S_i p_i^{n}; \hspace{-2.3cm} & \hfill \text{(local)} \label{step:dCG6} \\
			& \textstyle{  \sigma^{n} = \sum_{i \in \mathcal{S}} \sigma_i^{n} } \hspace{-1cm} & \text{ (scalar global sums)} \label{step:dCG1}\\
			&\textstyle{ r^{n+1}_i = r_i^n - \frac{\eta^{n}}{\sigma^{n}} \sum_{j \in \mathcal  N_i} I_{ij} u_j^n}\hspace{-2cm} & \text{ (neighbor-to-neighbor)} \label{step:dCG2} \\
			&  \textstyle{\eta^{n+1}_i= r_i^{n + 1\top } \Lambda_i^{-1} r_i^{n+ 1} } & \text{(local)} \label{step:dCG3}\\
			& \textstyle{\eta^{n+1}= \sum_{i \in \mathcal{S}} \eta_i^{n+1} } \hspace{-1.2cm} &\hfill \text{ (scalar global sum)} \label{step:dCG4} \\
			%		&  \hspace{-1.8cm}&  \text{(local)} \label{step:dCG3} \\
			&\textstyle{p_i^{n+1}= r_i^{n+1} + \frac{\eta^{n+1}}{\eta^n} p_i^n }; \;\; \textstyle{\lambda_i^{n+1} = \lambda_i^n + \frac{\eta^{n}}{\sigma^n} p_i} \hspace{-2cm} &\text{(local)} \label{step:dCG5} \\
			&n \leftarrow n+1 & \notag
		\end{align}
	\end{subequations} \label{alg:d-CG}
	\vspace{-.5cm}
\end{algorithm}

	Observe that Algorithm~\ref{alg:d-CG} comprises basic operations: steps \eqref{step:dCG3}, \eqref{step:dCG5} and \eqref{step:dCG6} are purely local steps, which compute matrix-vector products and scalar divisions.
	In each subsystem $i \in \mathcal{S}$, scalars $(\sigma_i,\eta_i)$ and vectors $(r_i,u_i,\lambda_i,p_i)$ have to be maintained locally.
	\eqref{step:dCG2}~is a decentralized step, which requires neighbor-to-neighbor communication only and can be seen as an average computation between neighbors. 
	The steps \eqref{step:dCG2} and \eqref{step:dCG4} require the communication of one scalar value per subsystem globally.
	Hence, most steps are local and require only neighbor-to-neighbor communication. The two global sums can simply be implemented via broadcasting or a round-robin protocol.
	Thus, Algorithm~\ref{alg:d-CG} is essentially decentralized.
	
	\subsection{$S$ is positive definite under \autoref{ass:basic}} \label{sec:posDefProof}
		\begin{lem}
			Suppose \autoref{ass:basic} holds. Then $S$ from \eqref{eq:schurComp} is positive definite. \hfill $\square$
		\end{lem}
		\begin{proof}
			From \eqref{eq:schurComp} we have $S = \sum_{i\in \mathcal{S}} \tilde{A}_i (\nabla F_i^\delta)^{-1} \tilde{A}_i^\top$. Observe that $S = \tilde{A} (\nabla F^\delta)^{-1} \tilde{A}^\top$, where ${\tilde{A}\doteq (\tilde{A}_1,\dots,\tilde{A}_{|\mathcal{S}|})}$ and $(\nabla F^\delta)^{-1} \doteq \text{diag}_{i\in\mathcal{S}}((\nabla F_i^\delta)^{-1})$. Because $\nabla_{xx} L_i$ is positive definite and because of LICQ in \autoref{ass:basic}, it follows from \cite[Lem. 16.1]{Nocedal2006} that the matrix $\nabla F_i^\delta$ is invertible. From the definition of $(\nabla F^\delta)^{-1}$ and the symmetry of $\nabla F_i^\delta$ it follows that $(\nabla F^\delta)^{-1}$ is symmetric and regular. Eigendecomposition hence gives ${(\nabla F^\delta)^{-1} = Q \Lambda^{1/2} (\Lambda^{1/2})^\top Q^\top}$ such that $S = B B^\top$, where $B \doteq \tilde{A} Q \Lambda^{1/2}$. LICQ ensures that $\tilde{A}$ and therefore $B$ have full row rank such that positive definiteness of $S$ follows from \cite[Thm 7.2.7]{Horn2013}.
		\end{proof}

	\renewcommand*{\bibfont}{\footnotesize}
	\footnotesize
	\printbibliography
	
\end{document}